\documentclass[12pt,reqno]{amsart}

\usepackage{tikz}
\usepackage{xcolor}
\usepackage{latexsym,amsmath,amssymb,epic}
\usepackage{latexsym,amssymb,epic}
\usepackage{amsmath,amsthm}
\usepackage{color}
\usepackage[left=3.5cm,right=3.5cm]{geometry}
\usepackage{etoolbox}
\usepackage{enumitem}
\usepackage[noadjust]{cite}
\usepackage[pagebackref]{hyperref}

\usepackage{hyperref}
\hypersetup{
	colorlinks=true, 
	linktoc=all, 
	linkcolor=blue} 

\allowdisplaybreaks[4]

\numberwithin{equation}{section}

\theoremstyle{theorem}
\newtheorem{theorem}{Theorem}[section]
\newtheorem*{theorem*}{Theorem}

\newtheorem{lemma}[theorem]{Lemma}

\newtheorem{problem}[theorem]{Problem}

\theoremstyle{definition}

\newtheorem*{example*}{Example}

\newtheorem*{conjecture*}{Conjecture}
\newtheorem{remark}[theorem]{Remark}
\newtheorem*{remark*}{Remark}
\newtheorem*{remarks*}{Remarks}

\makeatletter
\patchcmd{\@settitle}{\uppercasenonmath}{\boldmath\uppercasenonmath}{}{}
\patchcmd{\section}{\scshape}{\bfseries\boldmath}{}{}
\patchcmd{\subsection}{\bfseries}{\bfseries\boldmath}{}{}
\renewcommand{\@secnumfont}{\bfseries}
\makeatother

\newcommand{\cX}{\mathcal{X}}
\newcommand{\cZ}{\mathcal{Z}}

\begin{document}

\title[Ramanujan's theta functions]
{Ramanujan's theta functions and internal congruences modulo arbitrary powers of $3$}

\author[S. Chern]{Shane Chern}
\address[S. Chern]{Department of Mathematics and Statistics, Dalhousie University,
Halifax, NS, B3H 4R2, Canada}
\email{chenxiaohang92@gmail.com}

\author[D. Tang]{Dazhao Tang}
\address[D. Tang]{School of Mathematical Sciences, Chongqing Normal University,
Chongqing 401331, P.R. China}
\email{dazhaotang@sina.com}

\date{}

\dedicatory{}



\maketitle

\begin{abstract}
In this work, we investigate internal congruences modulo arbitrary powers of $3$
for two functions arising from Ramanujan's classical theta functions $\varphi(q)$
and $\psi(q)$. By letting
\begin{align*}
\sum_{n\ge 0} ph_3(n) q^n:=\dfrac{\varphi(-q^3)}{\varphi(-q)}\qquad\text{and}\qquad
\sum_{n\ge 0} ps_3(n) q^n:=\dfrac{\psi(q^3)}{\psi(q)},
\end{align*}
we prove that for any $m\ge 1$ and $n\ge 0$,
\begin{align*}
ph_3\big(3^{2m-1}n\big)\equiv ph_3\big(3^{2m+1}n\big)\pmod{3^{m+2}},
\end{align*}
and
\begin{align*}
ps_3{\left(3^{2m-1}n+\frac{3^{2m}-1}{4}\right)}\equiv
ps_3{\left(3^{2m+1}n+\frac{3^{2m+2}-1}{4}\right)}\pmod{3^{m+2}},
\end{align*}
thereby substantially generalizing the previous results of Bharadwaj et al.~and
Gireesh et al., respectively.
\end{abstract}

\section{Introduction}

The \emph{classical theta function} $f(a,b)$ was introduced by Ramanujan in his
\textit{Notebooks} \cite[p.~197, Entry 18]{Ram1957}:
\begin{align*}
f(a,b):=\sum_{k=-\infty}^\infty a^{k(k+1)/2}b^{k(k-1)/2},\qquad|ab|<1,
\end{align*}
and it also takes the product form by the \emph{Jacobi triple product identity}
\cite[p.~6, Eq. (0.15)]{Coo2017}:
\begin{align*}
f(a,b)=(-a,ab)_\infty (-b;ab)_\infty(ab;ab)_\infty,
\end{align*}
wherein the conventional \emph{$q$-Pochhammer symbol} is adopted:
\begin{align*}
(A;q)_\infty := \prod_{k\ge 0} (1-Aq^k).
\end{align*}
In particular, we are interested in the following two specializations:
\begin{alignat*}{4}
	\varphi(q) &\,:=\,\,&& \;f(q,q) &&\,=\,\,&& \sum_{k=-\infty}^\infty q^{k^2},\\
	\psi(q) &\,:=\,\,&& f(q,q^3) &&\,=\,\,&& \,\,\sum_{k\ge 0} q^{k(k+1)/2}.
\end{alignat*}
Alternatively,
\begin{align*}
\varphi(-q)=\frac{(q;q)_\infty^2}{(q^2;q^2)_\infty}\qquad\textrm{and}\qquad
\psi(q)=\frac{(q^2;q^2)_\infty^2}{(q;q)_\infty},
\end{align*}
where we choose $\varphi(-q)$ rather than $\varphi(q)$ in the former to
emphasize the similarity of the two product expressions.

A particularly notable topic in the world of $q$-series revolves around the
arithmetic properties of the coefficients $c(n)$ generated by
\begin{align*}
\sum_{n\ge 0} c(n)q^n := \prod_\delta (q^\delta;q^\delta)_\infty^{r_\delta}.
\end{align*}
Namely, we look for congruences of the form
\begin{align*}
	c(An+B) \equiv C \pmod{M},
\end{align*}
holding for any $n\ge 0$, in which the modulus $M$ and the parameters $A$, $B$
and $C$ are fixed, with $C$ usually being $0$.

The study of this problem was initiated by the celebrated congruences modulo
$5$, $7$ and $11$ due to Ramanujan \cite{Ram1919,Ram1921} for the
\emph{partition function} $p(n)$, which counts the number of integer partitions
of a natural number $n$ and has the generating function \cite{And1998}:
\begin{align*}
	\sum_{n\ge 0} p(n)q^n = \frac{1}{(q;q)_\infty}.
\end{align*}
Such congruences were later extended to moduli of an arbitrary power of $5$,
$7$ and $11$ as conjectured by Ramanujan \cite{Ram1988}: For
$\ell\in\{5,7,11\}$ and $\alpha\ge 1$,
\begin{align*}
p\big(\ell^\alpha n+\delta_{\alpha,\ell}\big)\equiv\begin{cases}
0 \pmod{\ell^\alpha} &\quad\ell=5,11,\\
0 \pmod{7^{\lceil\frac{\alpha+1}{2}\rceil}} &\quad\ell=7,\end{cases}
\end{align*}
with $0\le \delta_{\alpha,\ell} \le \ell^\alpha-1$ being such that
\begin{align*}
24\delta_{\alpha,\ell}\equiv1\pmod{\ell^\alpha}.
\end{align*}
Here Watson \cite{Wat1938} proved the cases of powers of $5$ and $7$, while
Atkin \cite{Atk1967} confirmed the case of powers of $11$.

In the meantime, we are also interested in \emph{internal} congruences of the form
\begin{align*}
c(An+B) \equiv c(A'n+B')\pmod{M},
\end{align*}
and we expect that the above two quantities are \emph{not} congruent to a fixed
number modulo $M$ for all $n$, so as to make this relation more nontrivial. In
many cases, the sequence $A'n+B'$ is rendered as a subsequence of $An+B$, and
such an internal congruence usually allows us to derive an infinite family of
congruences under the fixed modulus $M$. However, internal congruences modulo
an arbitrary power of a number are not widely recognized.

In this work, the first object of our interest is
\begin{align}\label{eq:ph-gf}
\sum_{n\ge 0}ph_3(n)q^n :=\dfrac{\varphi(-q^3)}{\varphi(-q)}.
\end{align}
From a partition-theoretic perspective, $ph_3(n)$ counts the number of
$3$-regular overpartitions of $n$ as introduced by Lovejoy \cite{Lov2003}, and
this function is usually written as $\overline{A}_3(n)$ in the literature. In
2018, Bharadwaj, Hemanthkumar and Naika \cite{BHN2018} established the
following internal congruences:
\begin{align*}
ph_3(27n) &\equiv ph_3(3n)\pmod{27},\\
ph_3(243n) &\equiv ph_3(27n)\pmod{81}.
\end{align*}
Herein, we offer a substantial generalization by extending the modulus to an
arbitrary power of $3$:

\begin{theorem}\label{th:ph-cong}
For any $m\ge 1$ and $n\ge 0$,
\begin{align}\label{eq:ph-cong}
ph_3\big(3^{2m+1}n\big)\equiv ph_3\big(3^{2m-1}n\big) \pmod{3^{m+2}}.
\end{align}
\end{theorem}

Similar to \eqref{eq:ph-gf}, we shall also consider
\begin{align}\label{eq:ps-gf}
	\sum_{n\ge 0} ps_3(n) q^n:= \frac{\psi(q^3)}{\psi(q)}.
\end{align}
The function $ps_3(n)$ is closely tied with the $\operatorname{pod}_3(n)$
function, which counts the number of partitions of $n$ into non-multiples of
$3$ in which the odd parts are distinct, through the relation
$ps_3(n)=(-1)^n\operatorname{pod}_3(n)$. According to the work of Gireesh,
Hirschhorn and Naika \cite{GHN2017}, the following internal congruences are
true:
\begin{align*}
ph_3(27n+20) &\equiv ph_3(3n+2)\pmod{27},\\
ph_3(243n+182) &\equiv ph_3(27n+20)\pmod{81}.
\end{align*}
In the same vein, the above two congruences will be extended in this work:

\begin{theorem}\label{th:ps-cong}
For any $m\ge 1$ and $n\ge 0$,
\begin{align}\label{eq:ps-cong}
ps_3{\left(3^{2m+1}n+\frac{3^{2m+2}-1}{4}\right)}
\equiv ps_3{\left(3^{2m-1}n+\frac{3^{2m}-1}{4}\right)}\pmod{3^{m+2}}.
\end{align}
\end{theorem}

It is remarkable (and indeed unexpected!) that once we have proved
\eqref{eq:ph-cong} using the strategy in Sect.~\ref{sec:ph-cong-proof}, the
congruence \eqref{eq:ps-cong} follows automatically, as shown in
Sect.~\ref{sec:ps-cong-proof}. However, an interesting fact is that, in regard
to the opposite direction, there likely exists an insuperable obstacle, and we
will present a concrete discussion in Remark \ref{rmk:ps-to-ph}.

The remainder of this paper is organized as follows. First, in
Sect.~\ref{sec:phi-notation} we introduce required auxiliary series associated
with $\varphi(-q)$ and establish their corresponding modular equations. Then
Sect.~\ref{sec:3-adic} is devoted to an initial $3$-adic analysis. With such
preliminary knowledge, we are about to prove the internal congruences for
$ph_3(n)$ in Sect.~\ref{sec:ph-cong-proof}. Finally, in
Sect.~\ref{sec:ps-cong-proof} we show how the internal congruences for
$ps_3(n)$ will automatically hold by constructing a different set of auxiliary
functions involving $\psi(q)$, which, surprisingly, possess the same modular
equations.

\section{Auxiliary functions and modular equations}\label{sec:phi-notation}

Let us write
\begin{align}
F(q):=\dfrac{\varphi(-q^3)}{\varphi(-q)}.
\end{align}
In the meantime, we define the following two auxiliary functions:
\begin{align}\label{eq:gamma-def}
\gamma=\gamma(q):=\frac{F(q)}{F(q^9)},
\end{align}
and
\begin{align}\label{eq:xi-def}
\xi=\xi(q):=\dfrac{\varphi(-q^9)}{\varphi(-q)}.
\end{align}
Finally, let $U$ be the \emph{unitizing operator of degree three}, given by
\begin{align*}
U{\left(\sum_n a_n q^n\right)}:= \sum_n a_{3n} q^n.
\end{align*}

In this section, our main objective is the following:

\begin{theorem}\label{th:xi-poly}
For every $i\ge 0$, both $U\big(\xi^i\big)$ and $U\big(\gamma \xi^i\big)$ can be
expressed as a polynomial in $\mathbb{Z}[\xi]$. In particular, if we write
\begin{align}\label{eq:X-coeff}
U\big(\xi^i\big) = \sum_{j} X_{i}(j)\xi^j,
\end{align}
then $X_i(j)=0$ whenever $0\le j< \left\lceil \frac{i}{3}\right\rceil$, where the ceiling
function $\lceil x\rceil$ denotes the least integer greater than or equal to
$x$.
\end{theorem}

We first focus on $U\big(\xi^i\big)$, and we shall start with the following
initial cases.

\begin{lemma}\label{le:xi-1-3}
We have
\begin{align}
U\big(\xi\big) &=\xi-3\xi^2+3\xi^3,\label{eq:xi1}\\
U\big(\xi^2\big) &=-2\xi+9\xi^2-24\xi^3+45\xi^4
-54\xi^5+27\xi^6,\label{eq:xi2}\\
U\big(\xi^3\big) &= \xi - 12 \xi^2 + 66 \xi^3 - 216 \xi^4 + 486 \xi^5 - 810 \xi^6\notag\\
&\quad + 972 \xi^7 - 729 \xi^8 + 243 \xi^9.\label{eq:xi3}
\end{align}
\end{lemma}

\begin{proof}
These relations can be shown by a direct cusp analysis as we note that $\xi$
is a modular function on the classical modular curve $X_0(18)$, which has
genus $0$. Alternatively, we may perform a more automated proof by first
utilizing Smoot's \textsf{Mathematica} implementation \texttt{RaduRK}
\cite{Smo2021} of the Radu--Kolberg algorithm \cite{Rad2015} to express each
of $U\big(\xi\big)$, $U\big(\xi^2\big)$ and $U\big(\xi^3\big)$ in terms of a
Hauptmodul on $X_0(18)$, and then applying Garvan's \textsf{Maple} package
\texttt{ETA} \cite{Gar1999}, which allows us to certify the equality of two
linear combinations of eta-products. See the proof of
\cite[Theorem 2.3]{CS2023} for a detailed instance.
\end{proof}

Now we move on to general $U\big(\xi^i\big)$. For brevity, we write for $i\ge 0$,
\begin{align}
\cX_i := U\big(\xi^i\big).
\end{align}

\begin{lemma}\label{le:cX-rec}
For any $i\ge 3$,
\begin{align}\label{eq:cX-rec}
\cX_i=\big(\xi-3\xi^2+3\xi^3\big)\big(3\cX_{i-1}-3\cX_{i-2}+\cX_{i-3}\big).
\end{align}
\end{lemma}

\begin{proof}
Let $\omega:=e^{\frac{2\pi i}{3}}$ be a primitive cubic root of unity and denote
for $k\in\{0,1,2\}$:
\begin{align*}
\xi_k:=\xi(\omega^k q).
\end{align*}
We start by noting that
\begin{align*}
	\sigma_1 &:= \xi_0 + \xi_1 + \xi_2\\
	&\ = 3U\big(\xi\big),\\
	\sigma_2 &:= \xi_0\xi_1 + \xi_1\xi_2 + \xi_2\xi_0\\
	&\ =\tfrac{1}{2}\big[\big(\xi_0 + \xi_1 + \xi_2\big)^2-\big(\xi_0^2 + \xi_1^2 + \xi_2^2\big)\big]\\
	&\ =\tfrac{1}{2}\big[9U\big(\xi\big)^2-3U\big(\xi^2\big)\big],\\
	\sigma_3 &:= \xi_0\xi_1\xi_2\\
	&\ = \tfrac{1}{6}\big[\big(\xi_0 + \xi_1 + \xi_2\big)^3-3\big(\xi_0 + \xi_1 + \xi_2\big)\big(\xi_0^2 + \xi_1^2 + \xi_2^2\big)+2\big(\xi_0^3 + \xi_1^3 + \xi_2^3\big)\big]\\
	&\ = \tfrac{1}{6}\big[27U\big(\xi\big)^3-27U\big(\xi\big)U\big(\xi^2\big)+6U\big(\xi^3\big)\big].
\end{align*}
These are instances of Newton's identities for elementary symmetric
functions \cite{Mea1992}. In light of Lemma \ref{le:xi-1-3}, it is clear
that all of $\sigma_1$, $\sigma_2$ and $\sigma_3$ are in
$\mathbb{Z}[\xi]$, and in particular,
\begin{align*}
\sigma_1 &=3\xi-9\xi^2+9\xi^3,\\
\sigma_2 &=3\xi-9\xi^2+9\xi^3,\\
\sigma_3 &=\xi-3\xi^2+3\xi^3.
\end{align*}
Next, we observe that $X=\xi_0$, $\xi_1$ and $\xi_2$ are the three roots of
\begin{align*}
\big(X-\xi_0\big)\big(X-\xi_1\big)\big(X-\xi_2\big)
=X^3-\sigma_1X^2+\sigma_2X-\sigma_3.
\end{align*}
So for $k\in\{0,1,2\}$,
\begin{align*}
\xi_k^3-\sigma_1\xi_k^2+\sigma_2\xi_k-\sigma_3=0,
\end{align*}
which further implies that for $i\ge 3$,
\begin{align*}
\xi_k^i=\sigma_1\xi_k^{i-1}-\sigma_2\xi_k^{i-2}+\sigma_3 \xi_k^{i-3}.
\end{align*}
Finally, we note that for $j\ge 0$,
\begin{align*}
\cX_j=U\big(\xi^j\big)=3\big(\xi(q)^j+\xi(\omega q)^j+\xi(\omega^2 q)^j\big)
=3\big(\xi_0^j+\xi_1^j+\xi_2^j\big).
\end{align*}
Hence, summing both sides of the previous relation over $k\in\{0,1,2\}$
yields the required result.
\end{proof}

Next, we establish a relation between $U\big(\gamma\xi^i\big)$ and
$U\big(\xi^{i+1}\big)$:

\begin{lemma}
For any $i\ge 0$,
\begin{align}\label{eq:U-gamma-cX}
U\big(\gamma \xi^i\big) = \xi^{-1} \cX_{i+1}.
\end{align}
\end{lemma}

\begin{proof}
In light of \eqref{eq:gamma-def} and \eqref{eq:xi-def},
\begin{align*}
\gamma=\dfrac{F(q)}{F(q^9)}=\dfrac{\varphi(-q^3)}{\varphi(-q)}
\dfrac{\varphi(-q^9)}{\varphi(-q^{27})}
=\dfrac{\varphi(-q^3)}{\varphi(-q^{27})}
\cdot\dfrac{\varphi(-q^9)}{\varphi(-q)}=\dfrac{\xi(q)}{\xi(q^3)}.
\end{align*}
Thus,
\begin{align*}
U\big(\gamma\xi^i\big)=U{\left(\dfrac{\xi(q)}{\xi(q^3)}\cdot\xi(q)^i\right)}
=\dfrac{1}{\xi(q)}\cdot U\big(\xi(q)^{i+1}\big),
	\end{align*}
	as required.
\end{proof}

Finally, we conclude our proof of Theorem \ref{th:xi-poly}.

\begin{proof}[Proof of Theorem \ref{th:xi-poly}]
It is plain that $\cX_0=1\in \mathbb{Z}[\xi]$. When $i\ge 1$, every $\cX_i$
can be expressed as a polynomial in $\mathbb{Z}[\xi]$ according to the
recurrence \eqref{eq:cX-rec} with the initial cases in Lemma \ref{le:xi-1-3}
kept in mind. In addition, by an inductive argument, it is also clear from
the recurrence \eqref{eq:cX-rec} that the minimal $\xi$-power in $\cX_i$ is
at least $\xi^{\lceil \frac{i}{3}\rceil}$, which, particularly, means that whenever
$i\ge 1$, the constant term in the $\xi$-polynomial representation of
$\cX_i$ vanishes. This fact, together with \eqref{eq:U-gamma-cX}, further
certifies that $U\big(\gamma \xi^{i-1}\big)\in \mathbb{Z}[\xi]$.
\end{proof}

\section{$3$-Adic analysis}\label{sec:3-adic}

Throughout, let $\nu(n)$ denote the \emph{$3$-adic evaluation} of $n$, which
is defined as the \emph{largest} nonnegative integer $\alpha$ such that
$3^\alpha\mid n$. Also, as a convention, we assume that $\nu(0)=\infty$.

Recall from Theorem \ref{th:xi-poly} that, as a polynomial in
$\mathbb{Z}[\xi]$, the $\xi$-powers in $\cX_i$ of degree lower than
\begin{align}
	d_i := \left\lceil \frac{i}{3}\right\rceil
\end{align}
all vanish. Now we perform the following $3$-adic analysis for the remaining
terms:

\begin{lemma}\label{le:3-adic-X}
Let the coefficients $X_i$ be as in \eqref{eq:X-coeff}. Then for any
$i\ge 1$,
\begin{align}\label{eq:3-adic-0}
\nu\big(X_i(d_i)\big)=0,
\end{align}
and further for any $j\ge 1$,
\begin{align}\label{eq:3-adic-j}
\nu\big(X_i(d_i+j)\big)\ge\left\lfloor\frac{j+1}{2}\right\rfloor.
\end{align}
\end{lemma}

\begin{proof}
In light of Lemma \ref{le:xi-1-3}, our statement holds for $i=1$, $2$ and
$3$. Now let us assume that the statement is true for $i=3I+1$, $3I+2$ and
$3I+3$ with a certain $I\ge 0$. We first certify the claim for $i=3I+4$. To
see this, we start with \eqref{eq:cX-rec}:
\begin{align*}
\cX_{3I+4}=\big(\xi-3\xi^2+3\xi^3\big)
\big(3\cX_{3I+3}-3\cX_{3I+2}+\cX_{3I+1}\big).
\end{align*}
Noting that
\begin{align*}
d_{3I+4} = I+2,
\end{align*}
and that
\begin{align*}
d_{3I+1}=d_{3I+2}=d_{3I+3}=I+1,
\end{align*}
we have
\begin{align*}
X_{3I+4}(d_{3I+4})=3X_{3I+3}(d_{3I+3})-3X_{3I+2}(d_{3I+2})
+X_{3I+1}(d_{3I+1}),
\end{align*}
which is not a multiple of $3$ since $X_{3I+1}(d_{3I+1})$ is not divisible
by $3$ as assumed. Thus, \eqref{eq:3-adic-0} holds for $i=3I+4$. Similarly,
\begin{align*}
X_{3I+4}(d_{3I+4}+1) &=3X_{3I+3}(d_{3I+3}+1)-3X_{3I+2}(d_{3I+2}+1)
+X_{3I+1}(d_{3I+1}+1)\\
 &\quad-9X_{3I+3}(d_{3I+3})+9X_{3I+2}(d_{3I+2})-3X_{3I+1}(d_{3I+1}).
\end{align*}
Since it is already assumed that
$\nu\big(X_{3I+1}(d_{3I+1}+1)\big)\ge \lfloor\frac{1+1}{2}\rfloor=1$, we
conclude that \eqref{eq:3-adic-j} holds for $j=1$. Finally, for $j\ge 2$,
\begin{align*}
	&X_{3I+4}(d_{3I+4}+j)\\
	&=3X_{3I+3}(d_{3I+3}+j) - 3X_{3I+2}(d_{3I+2}+j) + X_{3I+1}(d_{3I+1}+j)\\
	&\quad - 9X_{3I+3}(d_{3I+3}+j-1) + 9X_{3I+2}(d_{3I+2}+j-1) - 3X_{3I+1}(d_{3I+1}+j-1)\\
	&\quad + 9X_{3I+3}(d_{3I+3}+j-2) - 9X_{3I+2}(d_{3I+2}+j-2) + 3X_{3I+1}(d_{3I+1}+j-2).
\end{align*}
Recalling the assumption that
$\nu\big(X_i(d_i+j)\big)\ge\left\lfloor\frac{j+1}{2}\right\rfloor$ is valid
for every $j\ge 0$ when $i=3I+1$, $3I+2$ and $3I+3$, we are able to confirm
\eqref{eq:3-adic-j} for $j\ge 2$ since
\begin{align*}
\nu\big(X_{3I+4}(d_{3I+4}+j)\big)
 &\ge\min\Bigg\{\left\lfloor\dfrac{j+1}{2}\right\rfloor+1,
\left\lfloor\dfrac{j+1}{2}\right\rfloor+1,
\left\lfloor\dfrac{j+1}{2}\right\rfloor+0,\\
 &\qquad\qquad\left\lfloor\dfrac{j+0}{2}\right\rfloor+2,
\left\lfloor\dfrac{j+0}{2}\right\rfloor+2,
\left\lfloor\dfrac{j+0}{2}\right\rfloor+1,\\
 &\qquad\qquad\left\lfloor\dfrac{j-1}{2}\right\rfloor+2,
\left\lfloor\dfrac{j-1}{2}\right\rfloor+2,
\left\lfloor\dfrac{j-1}{2}\right\rfloor+1\Bigg\}\\
 &=\left\lfloor\dfrac{j+1}{2}\right\rfloor.
\end{align*}
Moreover, a similar analysis may be carried out for $i=3I+5$ and $3I+6$, and
thus our desired claim holds by induction.
\end{proof}

\begin{remark}\label{rmk:cX-min-power}
By recourse to \eqref{eq:3-adic-0}, it is immediate that for $i\ge 1$, the coefficient $X_i(d_i)$ never vanishes; otherwise, it is divisible by $3$, thereby contradicting \eqref{eq:3-adic-0}. Recalling the second part of
Theorem \ref{th:xi-poly}, we conclude that the minimal $\xi$-power in the polynomial representation of $\cX_i$ is \textbf{exactly} of degree
$d_i=\left\lceil \frac{i}{3}\right\rceil$.
\end{remark}

\section{Internal congruences for $ph_3(n)$}\label{sec:ph-cong-proof}

We are about to study the following family of series for $M\ge 1$:
\begin{align}
\Phi_M(q):=\begin{cases}
\displaystyle\dfrac{1}{F(q^3)}\sum_{n\ge 0}ph_3\big(3^{2m-1}n\big)q^n,
&\quad\textrm{if $M=2m-1$},\\
\displaystyle\dfrac{1}{F(q)}\sum_{n\ge 0}ph_3\big(3^{2m}n\big)q^n,
&\quad\textrm{if $M=2m$}.
\end{cases}
\end{align}
It is clear that for any $m\ge 1$,
\begin{align}
\Phi_{2m} &=U\big(\Phi_{2m-1}\big),\label{eq:Phi-even}\\
\Phi_{2m+1} &=U\big(\gamma\Phi_{2m}\big).\label{eq:Phi-odd}
\end{align}

Let us first consider a general setting, wherein the notation is as in
Sect.~\ref{sec:phi-notation}.

\begin{lemma}
For any $\Lambda\in \mathbb{Z}[\xi]$ such that the minimal $\xi$-power in
$\Lambda$ is no lower that $\xi^L$, then the following statements are true:
\begin{enumerate}
[label={(\roman*)~},leftmargin=*,labelsep=0cm,align=left,itemsep=6pt]
\item $U\big(\Lambda\big)$ is in $\mathbb{Z}[\xi]$, and the minimal
$\xi$-power in this polynomial representation is of degree at least
$\left\lceil\frac{L}{3}\right\rceil$;
		
\item $U\big(\gamma\Lambda\big)$ is in $\mathbb{Z}[\xi]$, and the minimal
$\xi$-power in this polynomial representation is of degree at least
$\left\lceil\frac{L-2}{3}\right\rceil$.
\end{enumerate}
\end{lemma}

\begin{proof}
Assume that
\begin{align*}
\Lambda=\sum_{\ell\ge L}a_\ell\xi^\ell.
\end{align*}
We have
\begin{align}\label{eq:U-Lambda}
U\big(\Lambda\big)=\sum_{\ell\ge L}a_\ell U\big(\xi^\ell\big)
=\sum_{\ell\ge L}a_\ell\cX_\ell.
\end{align}
Since the ceiling function is non-decreasing, the minimal $\xi$-power in
$U\big(\Lambda\big)$ is no lower than that in $\cX_L$, which is
$\xi^{\lceil\frac{L}{3}\rceil}$ in light of Remark \ref{rmk:cX-min-power}.
Next,
\begin{align}\label{eq:U-gamma*Lambda}
U\big(\gamma\Lambda\big)=\sum_{\ell\ge L}a_\ell U\big(\gamma\xi^\ell\big)
=\sum_{\ell\ge L}a_\ell\xi^{-1}\cX_{\ell+1},
\end{align}
where we have utilized \eqref{eq:U-gamma-cX}. Now the minimal $\xi$-power in
$U\big(\gamma\Lambda\big)$ is no lower than that in $\xi^{-1}\cX_{L+1}$,
while we note that $\left\lceil\frac{L+1}{3}\right\rceil-1
=\left\lceil\frac{L-2}{3}\right\rceil$.
\end{proof}

Now we observe that
\begin{align*}
\Phi_1=\dfrac{1}{F(q^3)}\sum_{n\ge 0}ph_3\big(3n+2\big)q^n
=U\big(\gamma\big).
\end{align*}
In light of \eqref{eq:xi1} and \eqref{eq:U-gamma-cX}, we have
\begin{align}\label{eq:Phi1}
\Phi_1=1-3\xi+3\xi^2.
\end{align}
By virtue of \eqref{eq:Phi-even}, we apply \eqref{eq:U-Lambda} and get
\begin{align}\label{eq:Phi2}
\Phi_2=1-9\xi+36\xi^2-81\xi^3+135\xi^4-162\xi^5+81\xi^6.
\end{align}
Furthermore, utilizing \eqref{eq:Phi-odd} and \eqref{eq:U-gamma*Lambda}
yields
\begin{align}\label{eq:Phi3}
\Phi_3
&=55-2163\xi+34509\xi^2-330318\xi^3+2227338\xi^4-11501919\xi^5\notag\\
&\quad+47744397\xi^6-164234952\xi^7+477601434\xi^8-1189266543\xi^9\notag\\
&\quad+2554873083\xi^{10}-4751141589\xi^{11}
+7644778785\xi^{12}-10594276335\xi^{13}\notag\\
&\quad+12526595811\xi^{14}-12440502369\xi^{15}
+10115979435\xi^{16}-6457008150\xi^{17}\notag\\
&\quad+3013270470\xi^{18}-903981141\xi^{19}+129140163\xi^{20}.
\end{align}
In general, by iterating the above process, the following claim is clear.

\begin{theorem}\label{th:Phi-poly}
For any $M\ge 1$, we have $\Phi_M \in \mathbb{Z}[\xi]$.
\end{theorem}

Now we recall that our objective is to prove the congruence
\eqref{eq:ph-cong}. Hence, we shall work on a new family of series for
$m\ge 1$:
\begin{align}
\widehat{\Phi}_m:=\dfrac{1}{F(q^3)}\sum_{n\ge 0}
{\left(ph_3\big(3^{2m+1}n\big)-ph_3\big(3^{2m-1}n\big)\right)}q^n.
\end{align}
In other words,
\begin{align}\label{eq:Phi-hat-exp-2}
\widehat{\Phi}_m=\Phi_{2m+1}-\Phi_{2m-1}.
\end{align}

It is clear that the congruence \eqref{eq:ph-cong} in Theorem
\ref{th:ph-cong} is a straightforward consequence of the following result:

\begin{theorem}
For any $m\ge 1$, we have $\widehat{\Phi}_m\in \mathbb{Z}[\xi]$.
Furthermore, if we write
\begin{align}\label{eq:Psi-hat-coeff}
\widehat{\Phi}_m=\sum_{k}C_m(k)\xi^k,
\end{align}
then for any $k\ge 0$,
\begin{align}\label{eq:nu-C-bound}
\nu\big(C_m(k)\big)\ge m+2+\left\lfloor\dfrac{k}{2}\right\rfloor.
\end{align}
\end{theorem}

\begin{proof}
We perform our proof by induction on $m$. Clearly, the $m=1$ case can be
confirmed by invoking \eqref{eq:Phi1} and \eqref{eq:Phi3}. Now assuming
the validity of the statement for a certain $m$, we shall prove
\eqref{eq:nu-C-bound} for $m+1$. Our starting point is the following
observation:
\begin{align}
\widehat{\Phi}_{m+1}=U\big(\gamma U\big(\widehat{\Phi}_m\big)\big).
\end{align}
This can be shown by sequentially acting the operators $U\big(\bullet\big)$
and $U\big(\gamma\cdot \bullet\big)$ to both sides of
\eqref{eq:Phi-hat-exp-2}, while for the right-hand side, we also require
\eqref{eq:Phi-even} and \eqref{eq:Phi-odd}.
	
Now, by \eqref{eq:X-coeff},
\begin{align*}
U\big(\widehat{\Phi}_m\big) &=C_m(0)+\sum_{k\ge 1}C_m(k)\cX_k\\
 &=C_m(0)+\sum_{k\ge 1}C_m(k)\sum_{j\ge 0}X_k(d_k+j)\xi^{d_k+j}.
\end{align*}
Let us write
\begin{align*}
U\big(\widehat{\Phi}_m\big)=\sum_{\ell}\widetilde{C}_m(\ell)\xi^\ell.
\end{align*}
We shall show that for any $\ell\ge 0$,
\begin{align}\label{eq:nu-C'-bound}
\nu\big(\widetilde{C}_m(\ell)\big) \ge m+2+\left\lfloor\frac{\ell}{2}\right\rfloor.
\end{align}
To see this, we note that $C_m(0)$ contributes to $\widetilde{C}_m(0)\xi^0$, while
according to our assumption,
\begin{align*}
\nu\big(C_m(0)\big)\ge m+2+\left\lfloor\frac{0}{2}\right\rfloor.
\end{align*}
Also, for the each term $C_m(k)X_k(d_k+j)\xi^{d_k+j}$ contributing to
$\widetilde{C}_m(d_k+j)\xi^{d_k+j}$, we note that $k\ge d_k=\left\lceil \frac{k}{3}\right\rceil$
whenever $k\ge 1$, and conclude that
\begin{align*}
\nu\big(C_m(k)X_k(d_k+j)\big)
 &=\nu\big(C_m(k)\big)+\nu\big(X_k(d_k+j)\big)\\
 &\ge\left(m+2+\left\lfloor\frac{k}{2}\right\rfloor\right)
+\left\lfloor\frac{j+1}{2}\right\rfloor\\
 &\ge m+2+\left\lfloor\frac{d_k}{2}\right\rfloor
+\left\lfloor\frac{j+1}{2}\right\rfloor\\
 &\ge m+2+\left\lfloor\frac{d_k+j}{2}\right\rfloor,
\end{align*}
where $\nu\big(C_m(k)\big)$ is bounded by the inductive hypothesis and
$\nu\big(X_k(d_k+j)\big)$ is bounded by Lemma \ref{le:3-adic-X}. Therefore,
\eqref{eq:nu-C'-bound} is established.
	
Next, we have
\begin{align*}
\widehat{\Phi}_{m+1} &=U\big(\gamma U\big(\widehat{\Phi}_m\big)\big)\\
 &=\sum_{\ell\ge 0}\widetilde{C}_m(\ell)U\big(\gamma \xi^\ell\big)\\
 &=\sum_{\ell\ge 0}\widetilde{C}_m(\ell)\xi^{-1}\cX_{\ell+1}\\
 &=\sum_{\ell\ge 0}\widetilde{C}_m(\ell)\sum_{j\ge 0}
X_{\ell+1}(d_{\ell+1}+j)\xi^{d_{\ell+1}+j-1}.
\end{align*}
Here we have applied \eqref{eq:U-gamma-cX} for the third equality. Taking
out the summands with $\ell=0$ and $1$ gives
\begin{align}\label{eq:Phi-hat-expand-m+1}
\widehat{\Phi}_{m+1} &=\widetilde{C}_m(0)\big(1-3\xi+3\xi^2\big)\notag\\
 &\quad+\widetilde{C}_m(1)
\big({-2}+9\xi-24\xi^2+45\xi^3-54\xi^4+27\xi^5\big)\notag\\
 &\quad+\sum_{\ell\ge 2}\widetilde{C}_m(\ell)
\sum_{j\ge 0}X_{\ell+1}(d_{\ell+1}+j)\xi^{d_{\ell+1}+j-1},
\end{align}
where $U\big(\gamma\big)=\xi^{-1}U\big(\xi\big)$ and
$U\big(\gamma\xi\big)=\xi^{-1}U\big(\xi^2\big)$ can be deduced by
\eqref{eq:xi1} and \eqref{eq:xi2}, respectively. Recall that
\begin{align*}
\widehat{\Phi}_{m+1}=\sum_{k}C_{m+1}(k)\xi^k.
\end{align*}
To show \eqref{eq:nu-C-bound} for the case of $m+1$, we need to prove that
\begin{align}\label{eq:nu-C-bound-m+1}
\nu\big(C_{m+1}(k)\big)\ge(m+1)+2+\left\lfloor\dfrac{k}{2}\right\rfloor
=m+3+\left\lfloor\frac{k}{2}\right\rfloor.
\end{align}
For the right-hand side of \eqref{eq:Phi-hat-expand-m+1}, the easier cases
are those with $\ell\ge 2$. Note that in such scenarios,
$\ell\ge d_{\ell+1}+1=\left\lceil\frac{\ell+1}{3}\right\rceil+1$. Thus,
\begin{align*}
\nu\big(\widetilde{C}_m(\ell)X_{\ell+1}(d_{\ell+1}+j)\big)
 &=\nu\big(\widetilde{C}_m(\ell)\big)+\nu\big(X_{\ell+1}(d_{\ell+1}+j)\big)\\
 &\ge \left(m+2+\left\lfloor\frac{\ell}{2}\right\rfloor\right)
+\left\lfloor\frac{j+1}{2}\right\rfloor\\
 &\ge m+2+\left\lfloor\frac{d_{\ell+1}+1}{2}\right\rfloor
+\left\lfloor\frac{j+1}{2}\right\rfloor\\
 &\ge m+2+\left\lfloor\frac{d_{\ell+1}+j+1}{2}\right\rfloor\\
 &=m+3+\left\lfloor\frac{d_{\ell+1}+j-1}{2}\right\rfloor.
\end{align*}
Meanwhile, a routine computation reveals that the contributions from
\begin{align*}
\widetilde{C}_m(0)\cdot (-3\xi)
\end{align*}
and
\begin{align*}
\widetilde{C}_m(1)\cdot(9\xi),\qquad \widetilde{C}_m(1)\cdot(45\xi^3),
\qquad \widetilde{C}_m(1)\cdot(-54\xi^4),\qquad \widetilde{C}_m(1)\cdot(27\xi^5)
\end{align*}
all match with \eqref{eq:nu-C-bound-m+1}. Hence, it will suffice to verify
that
\begin{align}
\nu\big(\widetilde{C}_m(0)-2\widetilde{C}_m(1)\big)
 &\ge m+3+\left\lfloor\dfrac{0}{2}\right\rfloor,\label{eq:nu-C'-extra-0}\\
\nu\big(3\widetilde{C}_m(0)-24\widetilde{C}_m(1)\big)
 &\ge m+3+\left\lfloor\dfrac{2}{2}\right\rfloor.\label{eq:nu-C'-extra-2}
\end{align}
We begin by observing from \eqref{eq:nu-C'-bound} that
\begin{align*}
\nu\big(3\widetilde{C}_m(0)-24\widetilde{C}_m(1)\big)\ge1+(m+2)=m+3.
\end{align*}
Thus, modulo $3^{m+3}$,
\begin{align*}
\widehat{\Phi}_{m+1}(q)\equiv \widetilde{C}_m(0)-2\widetilde{C}_m(1)\pmod{3^{m+3}}.
\end{align*}
In particular,
\begin{align}\label{eq:Phi-hat-0-cong-1}
\widehat{\Phi}_{m+1}(0)\equiv \widetilde{C}_m(0)-2\widetilde{C}_m(1)\pmod{3^{m+3}}.
\end{align}
Meanwhile, we note that the constant term in
\begin{align*}
\widehat{\Phi}_{m+1}(q)=\dfrac{1}{F(q^3)}\sum_{n\ge 0}
\left(ph_3\big(3^{2m+3}n\big)-ph_3\big(3^{2m+1}n\big)\right)q^n
\end{align*}
vanishes since
\begin{align*}
ph_3\big(3^{2m+3}0\big)-ph_3\big(3^{2m+1}0\big)=ph_3(0)-ph_3(0)=0.
\end{align*}
Thus,
\begin{align}\label{eq:Phi-hat-0-cong-2}
\widehat{\Phi}_{m+1}(0) = 0.
\end{align}
Combining \eqref{eq:Phi-hat-0-cong-1} and \eqref{eq:Phi-hat-0-cong-2} gives
\begin{align}\label{eq:Phi-hat-0-cong-1&2}
\widetilde{C}_m(0)-2\widetilde{C}_m(1) \equiv 0 \pmod{3^{m+3}}.
\end{align}
This yields \eqref{eq:nu-C'-extra-0}. In the meantime, recalling that
$\widetilde{C}_m(1)\equiv0\pmod{3^{m+2}}$ by \eqref{eq:nu-C'-bound}, we have
$-6\widetilde{C}_m(1)\equiv 0 \pmod{3^{m+3}}$, so that
\begin{align*}
\widetilde{C}_m(0)-8\widetilde{C}_m(1)\equiv0\pmod{3^{m+3}},
\end{align*}
or equivalently,
\begin{align}
3\widetilde{C}_m(0)-24\widetilde{C}_m(1) \equiv 0 \pmod{3^{m+4}}.
\end{align}
Now \eqref{eq:nu-C'-extra-2} is also confirmed. Putting the above
arguments together, we arrive at \eqref{eq:nu-C-bound-m+1}, and thus
close the requested inductive step.
\end{proof}

\section{Internal congruences for $ps_3(n)$}\label{sec:ps-cong-proof}

Herein, we define
\begin{align}
G(q):=\dfrac{\psi(q^3)}{\psi(q)}.
\end{align}
and introduce another set of auxiliary functions:
\begin{align}\label{eq:delta-def}
\delta=\delta(q):= q^{-2}\frac{G(q)}{G(q^9)},
\end{align}
and
\begin{align}\label{eq:zeta-def}
\zeta=\zeta(q):= q\frac{\psi(q^9)}{\psi(q)}.
\end{align}
Furthermore, for $i\ge 0$, define
\begin{align}
\cZ_i:= U\big(\zeta^i\big).
\end{align}

We have the following result that exhibits the resemblance between
$U\big(\zeta^i\big)$ and $U\big(\xi^i\big)$, as well as between
$U\big(\delta\zeta^i\big)$ and $U\big(\gamma\xi^i\big)$:

\begin{theorem}
For every $i\ge 0$,
\begin{align}\label{eq:Z-coeff}
\cZ_i = \sum_{j} X_{i}(j)\zeta^j \in \mathbb{Z}[\zeta],
\end{align}
where the coefficients $X_i(j)$ are \textbf{identical} to those given in
\eqref{eq:X-coeff} for $\cX_i$. In addition,
\begin{align}\label{eq:U-delta-cZ}
U\big(\delta \zeta^i\big) = \zeta^{-1} \cZ_{i+1}.
\end{align}
\end{theorem}

\begin{proof}
Analogous to Lemma \ref{le:xi-1-3}, we have the following initial
evaluations for $\cZ_i$ by either a cusp analysis or an automated
computer-aided verification:
\begin{align}
	\cZ_1 &= \zeta - 3 \zeta^2 + 3 \zeta^3,\label{eq:zeta1}\\
	\cZ_2 &= -2 \zeta + 9 \zeta^2 - 24 \zeta^3 + 45 \zeta^4 - 54 \zeta^5 + 27 \zeta^6,\label{eq:zeta2}\\
	\cZ_3 &= \zeta - 12 \zeta^2 + 66 \zeta^3 - 216 \zeta^4 + 486 \zeta^5 - 810 \zeta^6\notag\\
	&\quad + 972 \zeta^7 - 729 \zeta^8 + 243 \zeta^9.\label{eq:zeta3}
\end{align}
Particularly, these polynomials on the right are identical to those in
\eqref{eq:xi1}, \eqref{eq:xi2} and \eqref{eq:xi3} after replacing
$\zeta$ with $\xi$. We may continue to copy the argument built on
Newton's identities for Lemma \ref{le:cX-rec}, and then
\eqref{eq:Z-coeff} becomes plain. Finally, note that
\begin{align*}
\delta=q^{-2}\dfrac{G(q)}{G(q^9)}=q^{-2}\dfrac{\psi(q^3)}{\psi(q)}
\dfrac{\psi(q^9)}{\psi(q^{27})}=q^{-3}\dfrac{\psi(q^3)}{\psi(q^{27})}
\cdot q\dfrac{\psi(q^9)}{\psi(q)}=\dfrac{\zeta(q)}{\zeta(q^3)},
\end{align*}
which immediately indicates \eqref{eq:U-delta-cZ}.
\end{proof}

Define, for $M\ge 1$:
\begin{align}
\Psi_M(q):=\begin{cases}
\displaystyle\dfrac{1}{G(q^3)}\sum_{n\ge 0}
ps_3{\left(3^{2m-1}n+\frac{3^{2m}-1}{4}\right)}q^n,
 &\quad\textrm{if $M=2m-1$},\\
\displaystyle\dfrac{1}{G(q)}\sum_{n\ge 0}
ps_3{\left(3^{2m}n+\frac{3^{2m}-1}{4}\right)}q^n,
 &\quad\textrm{if $M=2m$}.
\end{cases}
\end{align}
Then we also have that, for any $m\ge 1$,
\begin{align}
\Psi_{2m} &=U\big(\Psi_{2m-1} \big),\label{eq:Psi-even}\\
\Psi_{2m+1} &=U\big(\delta\Psi_{2m} \big).\label{eq:Psi-odd}
\end{align}
In a parallel way, the following result is true:

\begin{theorem}\label{th:Psi-poly}
For any $M\ge 1$, if we write
\begin{align*}
\Phi_M = \sum_{k} c_m(k) \xi^k \in \mathbb{Z}[\xi]
\end{align*}
according to Theorem \ref{th:Phi-poly}, then
\begin{align*}
\Psi_M = \sum_{k} c_m(k) \zeta^k \in \mathbb{Z}[\zeta]
\end{align*}
with the \textbf{same} coefficients in the two polynomial representations.
\end{theorem}

We continue to define, for $m\ge 1$,
\begin{align}
\widehat{\Psi}_m:=\dfrac{1}{G(q^3)}\sum_{n\ge 0}
\left(ps_3\big(3^{2m+1}n+\tfrac{3^{2m+2}-1}{4}\big)
-ps_3\big(3^{2m-1}n+\tfrac{3^{2m}-1}{4}\big)\right)q^n.
\end{align}
That is,
\begin{align}\label{eq:Psi-hat-exp-2}
\widehat{\Psi}_m = \Psi_{2m+1}-\Psi_{2m-1}.
\end{align}
The following analogy is also clear.

\begin{theorem}
For any $m\ge 1$,
\begin{align}
\widehat{\Psi}_m = \sum_{k} C_m(k) \zeta^k \in \mathbb{Z}[\zeta],
\end{align}
where the coefficients $C_m(k)$ are \textbf{identical} to those given
in \eqref{eq:Psi-hat-coeff} for $\widehat{\Phi}_m$.
\end{theorem}

Finally, invoking the $3$-adic evaluations for the coefficients $C_m(k)$
in \eqref{eq:nu-C-bound}, Theorem \ref{th:ps-cong} is immediately
established.

\begin{remark}\label{rmk:ps-to-ph}
Consulting the resemblance between the strategy for the proofs of
Theorems \ref{th:ph-cong} and \ref{th:ps-cong}, it is natural to ask if
we could proceed in the opposite direction by first presenting a direct
proof of the internal congruences for $ps_3(n)$. Clearly, the arguments
in Sect.~\ref{sec:ph-cong-proof} can be transplanted readily by working
instead on $\mathbb{Z}[\zeta]$. However, one issue occurs at the very
end. Namely, one crucial property for $\widehat{\Phi}_m$ we have
utilized is \eqref{eq:Phi-hat-0-cong-2}:
\begin{align*}
\widehat{\Phi}_{m}(0)=0.
\end{align*}
However, for $\widehat{\Psi}_m$, we only have
\begin{align*}
\widehat{\Psi}_{m}(0)=ps_3{\left(\frac{3^{2m+2}-1}{4}\right)}
-ps_3{\left(\frac{3^{2m}-1}{4}\right)},
\end{align*}
the value of which is indefinite. This means that we cannot have an
immediate conclusion of a congruence parallel to
\eqref{eq:Phi-hat-0-cong-1&2}. Herein, we leave a question to the
interested reader.
	
\begin{problem}
Find a \textbf{direct} proof of the following congruence:
\begin{align}
ps_3{\left(\frac{3^{2m}-1}{4}\right)}\equiv
ps_3{\left(\frac{3^{2m+2}-1}{4}\right)}\pmod{3^{m+2}},
\end{align}
whenever $m\geq1$.
\end{problem}
\end{remark}

\section*{Declarations}

\subsection*{Competing interests}

The authors declare that they have no known competing financial
interests or personal relationships that could have appeared to
influence the work reported in this paper.

\section*{Acknowledgements}

Dazhao Tang was partially supported by the National Natural Science
Foundation of China (No.~12201093), the Natural Science Foundation
Project of Chongqing CSTB (No.~CSTB2022NSCQ--MSX0387), the Science
and Technology Research Program of Chongqing Municipal Education
Commission (No.~KJQN202200509) and the Doctoral start-up research
Foundation (No.~21XLB038) of Chongqing Normal University.

\bibliographystyle{amsplain}

\end{document}